\documentclass[a4paper,12pt,reqno]{amsart}
\usepackage[margin=2cm]{geometry}
\usepackage{amsthm}
\theoremstyle{definition}
\newtheorem{theorem}{Theorem}[section]
\newtheorem{definition}[theorem]{Definition}
\newtheorem{corollary}[theorem]{Corollary}
\newtheorem{proposition}[theorem]{Proposition}
\newtheorem{remark}[theorem]{Remark}
\newtheorem{example}[theorem]{Example}
\newtheorem{lemma}[theorem]{Lemma}
\newtheorem{problem}[theorem]{Problem}

\usepackage{graphicx}
\usepackage{mathptmx}      
 \usepackage{float}
 \usepackage{amsmath,amssymb,enumerate}
\usepackage{hyperref}

\newcommand{\N}{\mathbb{N}}

\newcommand{\F}{\mathbb{F}_2}
\newcommand{\W}{\mathcal{W}}
\newcommand{\LL}{\mathcal{F}_2}
\newcommand{\A}{\mathcal{A}_2}
\renewcommand{\AA}{\widetilde{\mathcal{A}}_2}
\newcommand{\WW}{\F\langle \W \rangle}
\newcommand{\Sq}{Sq}
\renewcommand{\P}{\mathcal{P}}

\begin{document}

\title{The mod $2$ dual Steenrod algebra as a subalgebra of the mod 2 dual Leibniz-Hopf algebra}


\author{Ne\c{s}et Den\.{i}z Turgay}
\address{
Eastern Mediterranean University,
Gazimagusa, TRNC, Mersin 10, Turkey}
\email{Deniz\_Turgay@Yahoo.com}
\author{Shizuo Kaji}
\thanks{The second named author was partially supported by KAKENHI, Grant-in-Aid for Young
     Scientists (B) 26800043 and JSPS Postdoctoral Fellowships for Research Abroad.}
\address{Department of Mathematical Sciences,
Faculty of Science, Yamaguchi University, 1677-1, Yoshida, Yamaguchi 753-8512, Japan
              Tel.: +81-83-9335652\\
              Fax: +81-83-9335652}
\email{skaji@yamaguchi-u.ac.jp}

\begin{abstract}
The mod $2$ Steenrod algebra $\A$ can be defined as the quotient of the mod $2$ Leibniz--Hopf algebra $\LL$ by the Adem relations.
Dually, the mod $2$ dual Steenrod algebra $\A^*$ can be thought of as a sub-Hopf algebra
of the mod $2$ dual Leibniz--Hopf algebra $\LL^*$.
We study $\A^*$ and $\LL^*$ from this viewpoint
and give generalisations of some classical results in the literature.
\end{abstract}


\maketitle

\section{The mod $2$  Leibniz--Hopf algebra and its dual}\label{section1}
Let $\LL$ be the free associative algebra over $\F$ generated by
the indeterminates $S^1, S^2, S^3, \ldots$ of degree $|S^i|=i$.
We often denote the unit $1$ by $S^0$.
This algebra is equipped with a co-commutative co-product given by
\begin{equation}\label{coproduct}
\Delta(S^n) = \sum_{i=0}^n S^i \otimes S^{n-i},
\end{equation}
which makes it a graded connected Hopf algebra.
This algebra $\LL$ is often called the mod $2$ {Leibniz--Hopf algebra}.
As an $\F$-module, $\LL$ has the following canonical basis:
\[
\{ S^I:=S^{i_1}S^{i_2}\cdots S^{i_n} \mid
    I=(i_1,i_2,\ldots,i_n)\in \N^n, 0\le n <\infty\},
\]
where we regard $S^I=1$ when $n=0$.

Note that the integral counterpart of $\LL$ is called the Leibniz--Hopf algebra
 and is isomorphic to  the \emph{ring of non-commutative symmetric functions}
 (\cite{gelfand}) and the \emph{Solomon Descent algebra} (\cite{S}).
 Its graded dual is the ring of quasi-symmetric functions
with the outer co-product, which has been studied by Hazewinkel, Malvenuto, and Reutenauer in \cite{H0,H1,H2,H3,MR}.

The  mod $2$ Steenrod algebra $\A$ is defined to be the
quotient Hopf algebra of $\LL$ by
the ideal generated by the Adem relations:
\begin{equation}\label{adem}
S^i S^j - \sum_{k=0}^{\lfloor i/2 \rfloor} {{j-k-1}\choose{i-2k}} S^{i+j-k}S^k.
\end{equation}
Denote the quotient map by $\pi: \LL \to \A$ and $\Sq^i=\pi(S^i)$.
It is well-known (see, for example, \cite{SE}) that the \emph{admissible monomials}
\[
\{\Sq^J:=\Sq^{j_1}Sq^{j_2}\cdots Sq^{j_n} \mid
    J=(j_1,j_2,\ldots,j_n)\in \N_{>0}^n, 0\le n <\infty, j_{k-1} \ge 2 j_{k} \forall k\}
\]
form a module basis for $\A$.
We will adhere to this purely algebraic definition and will not use 
any other known facts about $\A$.

By taking the graded dual of $\pi$, we obtain the following inclusion of Hopf algebras
\[
\pi^*: \A^* \to \LL^*.
\]
$\LL^*$ is given a module basis $S_I$ dual to $S^I$, that is,
\[
 \langle S^{I'}, S_I \rangle= \begin{cases} 1 & (I=I') \\ 0 & (I\neq I') \end{cases}.
\]
Similarly, we have the dual basis $\{\Sq_J \mid J \text{ admissible} \}$
for $\A^*$ determined by
\[
 \langle \Sq^{J'}, \Sq_J \rangle = \begin{cases} 1 & (J=J') \\ 0 & (J\neq J') \end{cases}.
\]

The commutative product among the basis elements in $\LL^*$ is given by the
{\em overlapping shuffle product} (see \S \ref{osp}) and
the co-product is given by
\begin{equation}\label{co-product}
 \Delta(S_{a_1,\ldots,a_n})=S_{a_1,\ldots,a_n} \otimes 1 + 1 \otimes S_{a_1,\ldots,a_n}+
 \sum_{i=1}^{n-1} S_{a_1,\ldots,a_i} \otimes S_{a_{i+1},\ldots,a_n}.
\end{equation}

The purpose of this paper is to deduce some of the classical results on $\A^*$
and its generalisations by considering it as a subalgebra of $\LL^\ast$.
We are particularly interested in the following problems.
\begin{problem}\label{prblems}
\begin{enumerate}[(i)]
\item \label{prob:1} Determine the coefficients in
\begin{equation}\label{eq:pi^*}
 \pi^*(\Sq_J) = \sum_{I} C^I_J S_I
\end{equation}
for all admissible sequences $J$.
This is important since in the dual
it is equivalent to computing the coefficients of the Adem relations
\begin{equation}\label{eq:pi}
 \Sq^I = \sum_{J:\text{admissible}} C^I_J \Sq^J
\end{equation}
for all sequences $I$.
\item \label{prob:2}
 Give an expansion of the dual Milnor bases in terms of the dual admissible monomial bases, i.e.,
determine the coefficient $B^L_J$ in
\[
 \xi^L = \sum_{J:\text{admissible}} B^L_J \Sq_J,
\]
where $\xi_n=\Sq_{2^{n-1},2^{n-2}\cdots 2^1,2^0}$ and
$\xi^L=\xi_1^{l_1}\xi_2^{l_2}\cdots \xi_n^{l_n}$ for $L=(l_1,l_2,\ldots, l_n)$.
\item \label{prob:3}
Generalise Milnor's conjugation formula (\cite{Milnor}) in $\A^*$ to $\LL^*$.
The formula for $\A^*$ is:
\[
\chi(\xi_n)= \sum_\alpha \prod_{i=1}^{l(\alpha)} \xi_{\alpha(i)}^{2^{\sigma(i)}},
\]
where $\alpha=(\alpha(1)|\alpha(2)|\ldots|\alpha(l(\alpha))$ runs through all the compositions of the integer $n$ and
$\sigma(i)=\sum_{j=1}^{i-1} \alpha(j)$.
\end{enumerate}
\end{problem}
Several different methods are known for resolving
(\ref{prob:1}) and (\ref{prob:2}) (see for example, \cite{Monks,NDT}), 
but our argument (\S \ref{computation}) is new in that
it is purely combinatorial using
the overlapping shuffle product on $\LL^\ast$. 
We implemented our algorithm into a Maple code
\cite{code}.
In \S \ref{conjugation} we discuss the conjugation (or antipode) in $\LL^\ast$ and give an answer to (\ref{prob:3}).
Finally, we give an explicit duality between the conjugation invariants in $\LL$ and $\LL^\ast$
in \S \ref{sec:duality}.


\section{Overlapping Shuffle product}\label{osp}
We recall the definition of the overlapping shuffle product
(\cite[Section~2]{CrossleyHopf},\cite{H0}).
Let $\W $ be the set of finite sequences of natural numbers:
\[
\W  = \{ (i_1,i_2,\ldots,i_n) \mid 0\le n <\infty \}.
\]
Note that we allow the length $0$ sequence.
Consider the $\F$-module $\WW$ freely generated by $\W$.
For a sequence $I=(i_1,i_2,\ldots,i_n)$,
denote its tail partial sequence $(i_k,i_{k+1},\ldots,i_n)$ by $I_k$.
When $n<k$, we regard $I_k$ as the length $0$ sequence.
We use the convention
\[
(a_1,a_2,\ldots,a_k,(b_1,\ldots,b_i)+(c_1,\ldots,c_j)):=(a_1,a_2,\ldots,a_k,b_1,\ldots,b_i)+
(a_1,a_2,\ldots,a_k,c_1,\ldots,c_j).
\]
The {\em overlapping shuffle product} on $\WW$ is defined as follows:
\begin{definition}

For $A=(a_1,a_2,\ldots,a_n)$ and $B=(b_1,b_2,\ldots,b_m)$, define their product inductively by
\[
 A\cdot B:= \begin{cases} A & (m=0) \\ B & (n=0) \\
 \sum_{0\le i \le n} (a_1,\ldots, a_i, b_1, A_{i+1}\cdot B_2) + \sum_{1\le i \le n} 
 (a_1,\ldots,  a_i+b_1, A_{i+1}\cdot B_2) & (\text{otherwise}). \end{cases}
\]
The product on $\WW$ is defined by the linear extension of the above.
\end{definition}

We say a term in $A\cdot B$ is {\em $a$-first} if there exists $k$ such that $a_k$ goes\footnote{when calculated symbolically}
to an entry to the left of $b_k$ and $a_i$ goes to the same entry as $b_i$ (that is, the entry makes $a_i+b_i$) for all $i<k$.
For example, $(a_1+b_1,a_2,b_2,b_3,a_3)$ is $a$-first while $(a_1+b_1,b_2,a_2,a_3,b_3)$ is not.
Observe that
\begin{lemma}\label{eqlength}
For equal length sequences, we have
\[
(a_1,\ldots,a_n) \cdot (b_1,\ldots,b_n)= (a_1+b_1,\ldots,a_n+b_n) + Z+ \tau(Z),
\]
where $Z$ is a sum of $a$-first terms and $\tau$ flips the occurrence of $a_i$ and $b_i$ for all $i$.
In particular, the product is commutative.
\end{lemma}
\begin{example}
{\small
\begin{align*}
(a_1,a_2)\cdot (b_1,b_2)&=(a_1+b_1,a_2+b_2) \\
       &+ (a_1+b_1,a_2,b_2)+(a_1,b_1,a_2+b_2)+(a_1,b_1+a_2,b_2)
        +(a_1,a_2,b_1,b_2)+(a_1,b_1,a_2,b_2)+(a_1,b_1,b_2,a_2) \\
        &+(b_1+a_1,b_2,a_2)+(b_1,a_1,b_2+a_2)
        +(b_1,a_1+b_2,a_2)+(b_1,b_2,a_1,a_2)+(b_1,a_1,b_2,a_2)+(b_1,a_1,a_2,b_2),
\end{align*}
}
where the second line consists of $a$-first terms
and the third line is the $\tau$-image of the second line.
\end{example}

\begin{corollary}\label{osp-power}
For $A=(a_1,\ldots,a_n)$,
\[
 A\cdot A=(2a_1,\ldots,2a_n), \quad A^{2^m}=(2^m a_1,\ldots,2^m a_n).
\]
\end{corollary}
\begin{proof}
In this case, the flip map $\tau$ in Lemma \ref{eqlength} is the identity.
\end{proof}

It is easy to see from the duality relation 
$\langle S_I S_J, S^K \rangle=\langle S_I\otimes S_J, \Delta(S^K) \rangle$
that the product on $\LL^*$ dual to \eqref{coproduct} is
given by $S_I S_J= \sum_{K\in I\cdot J} S_K$.

\section{Dual Steenrod algebra as a sub-Hopf algebra of $\LL^*$}
To identify the image of the inclusion $\pi^*:\A^* \to \LL^*$, we prove some lemmas in this section.
Let $\xi_n=\Sq_{2^{n-1}, 2^{n-2}, \ldots, 2^0}$.
\begin{lemma}[{cf. \cite{CrossleyHopf,NDT}}]\label{pi-image}
 We have
\begin{eqnarray*}
\pi^*(\Sq_{2^n}) &=& S_{2^n}, \\
\pi^*(\xi_n) &=& S_{2^{n-1}, 2^{n-2}, \ldots, 2^0}.
\end{eqnarray*}
\end{lemma}
\begin{proof}
For the first equation,
we have to show that for any non-admissible sequence $I$, the right-hand side of
\[
 \Sq^I = \sum_{J:\text{admissible}} C^I_J \Sq^J
\]
does not contain $\Sq^{2^n}$.
If there exists such an $I$, we can assume it has length two, that is, $I=(i,j)$.
(Because the right-hand side is obtained by successively applying the length two relations.)
By the Adem relations in Equation \eqref{adem}, we have $i+j=2^n$ and
\[
1\equiv \binom{j-1}{i}\equiv \binom{2^n-1-i}{i} \mod 2.
\]
However, the binary expressions of $2^n-1-i$ and $i$ are complementary
and the binary expression of $2^n-1-i$ contains at least one digit with $0$.
Hence, by Lucas' Theorem, we have $\binom{2^n-1-i}{i}\equiv 0 \mod 2$; we arrive at a contradiction.

For the second equation, suppose that there exists an $I=(i,j)$
 such that  $i < 2j$ and
 \[
 \Sq^{i,j} = 
 \sum_{k=0}^{\lfloor i/2 \rfloor} {{j-k-1}\choose{i-2k}} \Sq^{i+j-k} \Sq^k
 \]
contains $\Sq^{2^{n-k}}$ or $\Sq^{2^{n-k}}\Sq^{2^{n-k-1}}$ as a summand.
The former case is already ruled out by the first equation.
For the latter case to happen, we should have
\[
 i+j=2^{n-k}+2^{n-k-1}, \lfloor i/2 \rfloor \ge 2^{n-k-1}.
\]
But this implies $j\le 2^{n-k-1}$ so $i\ge 2j$; we arrive at a contradiction.
\end{proof}

Put $\bar{\xi}_n=\pi^*(\xi_n) = S_{2^{n-1}, 2^{n-2}, \ldots, 2^0}$.
We denote by $\AA^*$ the subalgebra of $\LL^\ast$ generated by
$\{\bar{\xi}_n \mid 0<n \}$.
For a sequence $L=(l_1,l_2,\ldots,l_n)$ of non-negative integers,
we denote $\bar{\xi}_1^{l_1} \bar{\xi}_2^{l_2} \cdots \bar{\xi}_n^{l_n}$
by $\bar{\xi}^L$.
Then, the monomials $\bar{\xi}^L$ span $\AA^*$.
Now, we identify $\AA^*$ with $\mathrm{Im}(\pi^*)$.

Recall the definition of the {\em excess vector} of an admissible sequence $J=(j_1,j_2,\ldots,j_n)$:
\[
 \gamma(j_1,j_2,\ldots,j_n)=(j_1-2j_2, j_2-2j_3,\ldots,j_{n-1}-2j_n,j_n).
\]
This gives a bijection between admissible sequences and sequences of non-negative integers.
The inverse is given by
\[
\gamma^{-1}(l_1,l_2,\ldots,l_n) = (l_1+2l_2+2^2l_3+\cdots+2^{n-1}l_n,\ldots,l_{n-1}+2l_n,l_n).
\]
We put the right lexicographic order on $\W$, i.e.,
\[
 (a_1,a_2,\ldots,a_n) > (b_1,b_2,\ldots,b_m)
  \Leftrightarrow (n>m) \text{ or } (\exists k, a_k>b_k \text{ and } a_i=b_i \forall i>k).
\]
This induces an ordering on the basis elements $S_I$
which is compatible with the overlapping shuffle product.
Observe that the lowest term in the product $S_I \cdot S_{I'}$
for $I=(i_1,i_2,\ldots)$ and $I'=(i'_1,i'_2,\ldots)$ is
$S_{(i_1+i'_1,i_2+i'_2,\ldots)}$.

\begin{lemma}\label{triangularity}
For an admissible sequence $J$,
\[
 \langle \bar{\xi}^{\gamma(J)}, S^I \rangle =
  \begin{cases} 1 & (I=J) \\ 0 &  (I<J). \end{cases}
\]
\end{lemma}
\begin{proof}
We proceed by induction on $J=(j_1,\ldots,j_n)$.
Put $J'=(j_1-2^{n-1},j_2-2^{n-2},\ldots,j_n-2^0)$. Then by induction hypothesis,
\[
\bar{\xi}^{\gamma(J')}=S_{J'} + (\text{terms higher than } S_{J'}).
\]
It follows that
\begin{eqnarray*}
\bar{\xi}^{\gamma(J)} &=& \bar{\xi}^{\gamma(J')} \cdot \bar{\xi}_n\\
 &=&
(S_{J'} + (\text{terms higher than } S_{J'}))\cdot S_{2^{n-1},2^{n-2},\ldots,2^0} \\
&=& S_{J} + (\text{terms higher than } S_J).
\end{eqnarray*}
\end{proof}

By this upper-triangularity,
the monomials $\bar{\xi}^L$ are linearly independent
and we have
\begin{theorem}
\[
 \mathrm{Im}(\pi^*) = \AA^* = \F[\bar{\xi}_1,\bar{\xi}_2,\ldots,].
\]
\end{theorem}
\begin{proof}
By Lemma \ref{triangularity}
in each degree $\AA^*$ has the same dimension as $\A^*$
(the number of admissible sequences).
\end{proof}
This is nothing but the well-known fact:
\begin{corollary}[\cite{Milnor}]\label{A-triangularity}
\[
 \A^*=\F[\xi_1,\xi_2,\ldots,],
\]
where
\[
 \xi^{\gamma(J)}= \Sq_J + \text{(terms higher than $\Sq_J$)}.
\]
\end{corollary}

\section{Computation with $\pi^*$}\label{computation}
Recall from \cite[Section~4]{NDT} the linear left inverse $r\colon\LL^* \to \A^*$ of $\pi^*$:
\[
 r(S_I)=\begin{cases} Sq_I & (I:\text{admissible}) \\ 0 & (\text{otherwise}). \end{cases}
\]
For (\ref{prob:2}) of Problem \ref{prblems}, we can compute
\begin{equation}\label{pi^*}
 \xi^{(l_1,l_2,\ldots, l_n)}=r\pi^*(\xi^{(l_1,l_2,\ldots, l_n)})=r(\bar{\xi_1}^{l_1}\bar{\xi_2}^{l_2}\cdots \bar{\xi_n}^{l_n})
 =r((S_{2^0})^{l_1} (S_{2^1, 2^0})^{l_2}\cdots (S_{2^{n-1}, 2^{n-2}, \ldots, 2^0})^{l_n})
\end{equation}
and it reduces to computing admissible sequences occurring in the overlapping shuffle product.

For (\ref{prob:1}) of  Problem  \ref{prblems}, by Corollary \ref{A-triangularity} we have
\[
 \pi^*(\xi^{\gamma(J)})=\pi^*(\Sq_J + \text{(terms higher than $\Sq_J$)})
\]
and the left-hand side can be computed by the overlapping shuffle product.
Thus, we can compute inductively the coefficients $C^I_J$ in
\[
 \pi^*(\Sq_J) = \sum_{I} C^I_J S_I.
\]
We implemented the algorithm into a Maple code \cite{code}.

\begin{example}
We demonstrate the above algorithm in low degrees.
First, compute $\pi^*$-image of monomials $\xi^L$:
\begin{align*}
\pi^*(\xi_2^2) & = S_{2,1} S_{2,1}=S_{4,2}\\ 
\pi^*(\xi_1^3\xi_2) &=(S_3+S_{1,2}+S_{2,1})S_{2,1}\\
& =
S_{5, 1}+S_{4, 2}+S_{3, 3}+S_{2, 4}+S_{2, 3, 1}+S_{1, 4, 1}
+S_{3, 1, 2}+S_{2, 2, 2}+S_{1, 2, 3}+S_{2, 1, 2, 1}+S_{1, 2, 1, 2} \\
\pi^*(\xi_1^6) &= S_{6}+S_{4,2}+S_{2,4}.
\end{align*}
Taking $r$ on the both sides of equations, we obtain
\[
\xi_2^2 =\Sq_{4,2},\ \xi_1^3\xi_2 =\Sq_{5,1}+\Sq_{4,2}, \ \xi_1^6=\Sq_6+\Sq_{4,2}.
\]
Again taking $\pi^*$ on the both sides of the equations, we obtain
\begin{align*}
\pi^*(\Sq_{4,2}) &= S_{4,2} \\
\pi^*(\Sq_{5,1}+\Sq_{4,2}) &=
S_{5, 1}+S_{4, 2}+S_{3, 3}+S_{2, 4}+S_{2, 3, 1}+S_{1, 4, 1}
+S_{3, 1, 2}+S_{2, 2, 2}+S_{1, 2, 3}+S_{2, 1, 2, 1}+S_{1, 2, 1, 2} \\
\pi^*(\Sq_6+\Sq_{4,2}) &= S_{6}+S_{4,2}+S_{2,4}.
\end{align*}
Finally, by using the upper-triangularity, we obtain
\begin{align*}
\pi^*(\Sq_{4,2}) &= S_{4,2}    \\
\pi^*(\Sq_{5,1}) &=
S_{5, 1}+S_{3, 3}+S_{2, 4}+S_{2, 3, 1}+S_{1, 4, 1}+S_{3, 1, 2}+S_{2, 2, 2}+S_{1, 2, 3}+S_{2, 1, 2, 1}+S_{1, 2, 1, 2}  \\ 
\pi^*(\Sq_{6}) &= S_6+S_{2,4}.
\end{align*}
\end{example}

\section{Formula for the conjugation}\label{conjugation}
Any connected commutative or co-commutative  Hopf  algebra has a unique conjugation $\chi$ satisfying
\[
 \chi(1)=1, \ \chi(xy)=\chi(y)\chi(x), \ \chi^2(x)=x, \ \sum x'\chi(x'')=0,
\]
where $\Delta(x)=\sum x'\otimes x''$ and $\deg(x)>0$ (\cite{MM}).
The conjugation invariants in $\A^*$ is studied in \cite{CW}
because it is relevant to the commutativity of ring spectra \cite[Lecture~3]{Adams}.
The same problem in $\LL^*$ has been also studied in \cite{NS,CT}.
Here we investigate them through our point of view.

Since $\pi^\ast$ is a Hopf algebra homomorphism,
we have $\pi^*\circ \chi_{\A^\ast} = \chi_{\LL^*}\circ \pi^*$,
where $\chi_{\A^\ast}$ and $\chi_{\LL^*}$ denote the conjugation operations in $\A^\ast$ and  $\LL^*$ respectively.
For the module basis $S_I$ in $\LL^*$, the conjugation $\chi_{\LL^*}$ is calculated combinatorially.
\begin{definition}
The coarsening set $C(I)$ of a sequence $I=(i_1,\ldots,i_l)$
is defined recursively as
\[
 C(I) := \{ (i_1, I'), (i_1+i'_1,I'_2) \mid I'\in C((i_2,\ldots,i_l)) \} \text{ and } C((i))=\{ (i) \},
\]
where $I'_2$ is the tail partial sequence $(i'_2,\ldots,i'_{l'})$
of $I'=(i'_1,i'_2,\ldots,i'_{l'})$.
\end{definition}
\begin{example}
$C((a,b,c))=\{ (a,b,c), (a+b,c), (a,b+c), (a+b+c) \}$.
\end{example}

A formula for the conjugation operation in the dual Leibniz--Hopf algebra is given by Ehrenborg \cite[Proposition~3.4]{E}.
 We now give a simple proof for its mod $2$ reduction.
\begin{proposition}\label{chi}
\[
 \chi_{\LL^*}(S_I)= \sum_{I'\in C(I^{-1})} S_{I'},
\]
where $I^{-1}=(i_l,\ldots,i_1)$ is the reverse sequence of $I=(i_1,\ldots,i_l)$.
\end{proposition}
\begin{proof}
The conjugation is uniquely characterised by
\[
 \chi_{\LL^*}(1)=1, \sum x'\chi_{\LL^*}(x'')=0,
\]
where $\Delta(x)=\sum x'\otimes x''$ and $\deg(x)>0$.
We put $\chi'(S_I)=\sum_{I'\in C(I^{-1})} S_{I'}$ and show that it satisfies the above equations.
It is obvious that $\chi'(1)=1$.
Since the co-product is given in \eqref{co-product},
the second equation reads
\[
\sum_{k=0}^{l} S_{i_1,\ldots,i_k} \chi'(S_{i_{k+1},\ldots,i_n})=0 \qquad (\forall I=(i_1,i_2,\ldots,i_n)).
\]
We regard an element of $\WW$ with a finite subset of $\W$ in the obvious way.
 We investigate relation between coarsening and the overlapping shuffle product.
Define
\[
 C_k(I)
=\sum_{  I' \in C((I_{k+1})^{-1})} I' \cdot (i_1, \ldots, i_k).
\]
We observe\footnote{Here, 
we deal with sequences symbolically so that we avoid cancellations 
like $(i_3+i_2,i_1)+(i_3+i_1,i_2)=0$ when $i_1=i_2$.}
that $C(I^{-1}) \subset C_1(I)$
and $C'_1(I):=C_1(I) \setminus C(I^{-1})$ consists of those sequences that $i_1$ appears to the left of $i_2$.
In turn, $C'_1(I) \subset C_2(I)$ and $C'_2(I):=C_2(I) \setminus C'_1(I)$ consists of those 
sequences that $i_2$ appears to the left of $i_3$.
Continuing similarly, we obtain
\[
 C(I^{-1}) 
=\sum_{k=1}^l C_k(I).
\] 
 It follows that
\[
\chi'(S_I)=\sum_{k=1}^l \sum_{I' \in C((I_{k+1})^{-1})}
S_{(i_1, \ldots, i_k)}\cdot S_{I'}
=\sum_{k=0}^{l} S_{i_1,\ldots,i_k} \chi'(S_{i_{k+1},\ldots,i_n})-\chi'(S_{I})
\]
and $\sum_{k=0}^{l} S_{i_1,\ldots,i_k} \chi'(S_{i_{k+1},\ldots,i_n})=0$.
\end{proof}

We give another formula for $\chi_{\LL^*}(S_I)$.
\begin{definition}
For a sequence $a_1,a_2,\ldots,a_n$,
the {\em set of ordered block partitions} $\P(a_1,a_2,\ldots,a_n)$
consists of elements of the form
\[
 \beta=((a_1,a_2,\ldots,a_{i_1}) | (a_{i_1+1},\ldots,a_{i_2}) | \ldots |
 (a_{i_{l-1}+1},\ldots,a_{i_l})),
\]
where $1\le i_1<i_2<\cdots<i_l=n$. We denote $l(\beta)=l$ and
$\beta(k)=(a_{i_{k-1}+1},\ldots,a_{i_k})$.
Or inductively, we can define
\begin{equation}\label{ordered-partition}
\P(a_1,a_2,\ldots,a_n) = \bigcup_{k=1}^n 
\bigg\{ ((a_1,\ldots,a_k)| \beta) \bigg\rvert \beta\in \P(a_{k+1},a_{k+2},\ldots,a_n) \bigg\}.
\end{equation}
\end{definition}

\begin{theorem}\label{conjform}
\[
\chi_{\LL^*}(S_I)=\sum_{\beta\in \P(I)}
 \prod_{k=1}^{l(\beta)} S_{\beta(k)}.
\]
\end{theorem}
\begin{proof}
Let $I=(a_1,a_2,\ldots,a_n)$. Put
\[
\chi'(S_{a_1,a_2,\ldots,a_n})=
\sum_{\beta\in \P(a_1,a_2,\ldots,a_n)}
 \prod_{k=1}^{l(\beta)} S_{\beta(k)}
\]
and we check that
\[
 \chi'(1)=1,
 \sum_{k=0}^{n} S_{a_1,\ldots,a_k} \chi'(S_{a_{k+1},\ldots,a_n})=0.
\]
Then, by the uniqueness of the conjugation, we have $\chi_{\LL^*}=\chi'$.
The first assertion is trivial. For the second,
observe that by (\ref{ordered-partition})
\[
 \chi'(S_{a_1,a_2,\ldots,a_n})=
 \sum_{k=1}^n S_{a_1,\ldots,a_k}
 \left(\sum_{\beta\in \P(a_{k+1},a_{k+2},\ldots,a_n)}
 \prod_{j=1}^{l(\beta)} S_{\beta(j)} \right)
 =\sum_{k=1}^n S_{a_1,\ldots,a_k} \chi'(S_{a_{k+1},\ldots,a_n}).
\]
Hence, we have
\begin{align*}
 \sum_{k=0}^{n} S_{a_1,\ldots,a_k} \chi'(S_{a_{k+1},\ldots,a_n}) &=
 \chi'(S_{a_1,a_2,\ldots,a_n})+\sum_{k=1}^{n} S_{a_1,\ldots,a_k}  \chi'(S_{a_{k+1},\ldots,a_n}) \\
 & = 2\chi'(S_{a_1,a_2,\ldots,a_n}) \\
 &=0.
\end{align*}

\end{proof}

\begin{example}
\begin{align*}
\chi_{\LL^*}(S_{1,2,3}) &= S_{3,2,1}+S_{5,1}+S_{3,3}+S_{6} \\
   &= S_{1,2,3}+S_1 S_{2,3} +S_{1,2}S_3 + S_1 S_2 S_3.
\end{align*}
The first line is computed by Proposition \ref{chi}, and 
the second by Theorem \ref{conjform}.
\end{example}

Theorem \ref{conjform} can be thought of as a generalisation of
Milnor's conjugation formula in $\A^\ast$.
To see this, we first show a small lemma:
\begin{lemma}\label{pi-image-of-xi}
\[
 \bar{\xi}_n^{2^m}=(S_{2^{n-1}, 2^{n-2}, \ldots, 2^0})^{2^m}=S_{2^{n+m-1}, 2^{n+m-2}, \ldots, 2^m}.
\]
\end{lemma}
\begin{proof}
This is a direct consequence of Corollary \ref{osp-power}
 combined with Lemma \ref{pi-image}.
\end{proof}
\begin{corollary}[{\cite[Lemma 10]{Milnor}}]
\begin{equation}\label{milnor-conj}
\chi_{\A^\ast}(\xi_n)= \sum_\alpha \prod_{k=1}^{l(\alpha)} \xi_{\alpha(k)}^{2^{\sigma(k)}},
\end{equation}
where $\alpha=(\alpha(1)|\alpha(2)|\ldots|\alpha(l(\alpha))$ runs through all the compositions of the integer $n$ and
$\sigma(k)=\sum_{j=1}^{k-1} \alpha(j)$.
\end{corollary}
\begin{proof}
We apply the injection $\pi^*$ to the both sides of \eqref{milnor-conj}
and show that they coincide.
For the left-hand side, by Lemma \ref{pi-image-of-xi} we have
\[
\pi^*(\chi_{\A^\ast}(\xi_n))=\chi_{\LL^*}(\pi^*(\xi_n))=
\chi_{\LL^*}(S_{2^{n-1},2^{n-2},\ldots,2^0}).
\]
Since $\pi^*(\xi_{\alpha(k)}^{2^{\sigma(k)}})=S_{2^{\alpha(k)+\sigma(k)-1},\ldots,2^{\sigma(k)}}$
 by Lemma \ref{pi-image-of-xi}, we see
\[
\pi^* \left(\prod_{k=1}^{l(\alpha)} \xi_{\alpha(k)}^{2^{\sigma(k)}} \right)=
 S_{2^{n-1},\ldots,2^{n-\alpha(l(\alpha))}} \cdot
 S_{2^{n-1-\alpha(l(\alpha))},\ldots,2^{n-\alpha(l(\alpha))-\alpha(l(\alpha)-1)}}
 \cdots
 S_{2^{\alpha(1)-1},\ldots,2^0}.
\]
So when $\alpha$ ranges over all compositions of $n$,
we get all the ordered block partitions of the sequence $2^{n-1},2^{n-2},\ldots,2^0$.
The assertion follows from Theorem \ref{conjform}.
\end{proof}

\section{Duality between $\LL$ and $\LL^\ast$}\label{sec:duality}
In the previous section we discussed how to compute the conjugation in $\LL^\ast$.
Here, we relate the conjugation in $\LL$ with that in $\LL^\ast$
by using a self-duality of $\W$.
Denote $I \preceq I'$ if $I\in C(I')$.
We think of $I\in \W $ as a string of $1$'s separated by `$+$` and commas; 
$(\underbrace{1+1+\cdots+1}_{i_1}, \underbrace{1+1+\cdots+1}_{i_2},\ldots,\underbrace{1+1+\cdots+1}_{i_l})$.
\begin{definition}
We define the dual $\bar{I}\in \W $ of $I$ by switching $+$ and the commas.
\end{definition}
\begin{example}
For $I=(1,3,2)=(1,1+1+1,1+1)$, its dual is
\[ \bar{I}=(1+1,1,1+1,1)=(2,1,2,1). \]
\end{example}
It is easily seen that $\bar{\bar{I}}=I$ and $I \preceq I' \Leftrightarrow \bar{I} \succeq \bar{I}'$.
Extend the duality to one between $\LL$ and $\LL^*$ by
\[
 D(S^I) = S_{\bar{I}}, \quad D^{-1}(S_I) = S^{\bar{I}}.
\]

\begin{theorem}\label{duality}
We have $D\circ \chi_{\LL}= \chi_{\LL^\ast}\circ D$. In particular,
$f \in \LL$ is a conjugation invariant if and only if so is $\bar{f} \in \LL^*$.
\end{theorem}
\begin{proof}
We compute
\begin{align*}
D^{-1}\circ \chi_{\LL^\ast}\circ D ( S^I )
&= D^{-1}  \chi_{\LL^\ast}(S_{\bar{I}}) 
= D^{-1} (\sum_{I' \preceq(\bar{I})^{-1}} S_{I'}) 
= D^{-1} (\sum_{\bar{I}' \succeq I^{-1}} S_{I'}) \\
&= \sum_{\bar{I}'\succeq I^{-1}} S^{\bar{I}'} 
= \sum_{I'\succeq I^{-1}} S^{I'}.
\end{align*}
Put $\chi'(S^I)=\sum_{I'\succeq I^{-1}} S^{I'}$.
Then, one can check $\chi'(1)=1$ and $\sum x' \chi'(x'')=0$ for $\Delta x= \sum x'\otimes x''$
as in Proposition \ref{chi}.
Hence, by the uniqueness of the conjugation, we have $\chi'=\chi_{\LL}$.
\end{proof}
\begin{example}
$f=S^{1,1,2}+S^{2,1,1}+S^{1,1,1,1}$ is a $\chi_{\LL}$-invariant,
whilst $D(f)=S_{3,1}+S_{1,3}+S_4$ is a $\chi_{\LL^*}$-invariant.
\end{example}

\begin{remark}
The sub-module of the conjugation invariants in $\LL$ is $\ker(\chi_{\LL}-1)$ and
that in $\LL^\ast$ is $\ker(\chi_{\LL^*}-1)$.
The conjugations in $\LL$ and $\LL^*$ are dual to each other,
and hence, the linear map $\chi_{\LL}-1$ is
transpose to $\chi_{\LL^*}-1$ with the kernel of same dimension (\cite{NS}).
Theorem \ref{duality} gives more information by specifying
an explicit correspondence between their elements.
\end{remark}

\section*{Acknowledgements}
We would like to thank Stephen Theriault, Martin Crossley, and Carmen Rovi for their comments on the earlier version of this paper.

\end{document}